\documentclass[a4paper,10pt]{article}
\usepackage[utf8x]{inputenc}
\usepackage{amssymb,latexsym}
\usepackage{amsmath,amsthm}
\usepackage{graphicx}

\newcounter{case}
\newcounter{subcase}[case]
\renewcommand{\thesubcase}{\thecase\alph{subcase}}
\newcommand{\case}[1]{\medskip\noindent%
   \refstepcounter{case}\textbf{Case \thecase:} \ \textit{#1}\smallskip}
\newcommand{\subcase}[1]{\smallskip\noindent%
   \refstepcounter{subcase}\textbf{Case \thesubcase:} \ \textit{#1}\smallskip}

\newtheorem{theorem}{Theorem}[section]
\newtheorem{lemma}[theorem]{Lemma}

\newtheorem{conjecture}[theorem]{Conjecture}
\newtheorem{corollary}[theorem]{Corollary}
\newtheorem{remark}[theorem]{Remark}

\newtheorem{problem}{Open Problem}

\renewcommand{\ge}{\geqslant}
\renewcommand{\le}{\leqslant}

\newcommand\etal{~\emph{et al.}~}
\newcommand\defic{{\rm def}}

\def\eref#1{$(\ref{#1})$}

\def\sref#1{\S$\ref{#1}$}
\def\lref#1{Lemma~$\ref{#1}$}
\def\tref#1{Theorem~$\ref{#1}$}
\def\cref#1{Corollary~$\ref{#1}$}
\def\cjref#1{Conjecture~$\ref{#1}$}

\def\fref#1{Figure~$\ref{#1}$}

\usepackage{array,color,colortbl}

\newcommand{\R}{\mathbb{R}}
\newcommand{\cc}{\mathcal{C}}

\newcommand{\ci}{\mathcal{I}}
\newcommand{\ca}{\mathcal{A}}
\newcommand{\pp}[1]{\ensuremath{\mathcal{P}_{#1}}}
\newcommand{\tpp}[1]{\ensuremath{\mathcal{P}_{#1}'}}

\title{Multipartite hypergraphs achieving
equality in Ryser's conjecture\thanks{Research supported by 
ARC grant DP120100197, OTKA Grant NN.114614,
BSF grant 2018106 and ISF grant 1581/12.
}
}

\author{
Ron Aharoni\\
\small Department of Mathematics, Technion, Haifa 32000, Israel\\
\and
J\'anos Bar\'at\thanks{Current address: MTA-ELTE Geometric and Algebraic Combinatorics Research Group} \ 
and Ian M. Wanless\\
\small School of Mathematical Sciences, Monash University\\[-0.8ex]
\small Clayton, Vic 3800, Australia.
}

\date{}

\begin{document}

\maketitle

\begin{abstract}
  A famous conjecture of Ryser \cite{ryser} is that in an $r$-partite
  hypergraph the covering number is at most $r-1$ times the matching
  number.  If true, this is known to be sharp for $r$ for which there
  exists a projective plane of order $r-1$.  We show that the
  conjecture, if true, is also sharp for the smallest previously open
  value, namely $r=7$.  For $r\in\{6,7\}$, we find the minimal number
  $f(r)$ of edges in an intersecting $r$-partite hypergraph that has
  covering number at least $r-1$.  We find that $f(r)$ is achieved
  only by linear hypergraphs for $r\le5$, but that this is not the
  case for $r\in\{6,7\}$.  We also improve the general lower bound on
  $f(r)$, showing that $f(r)\ge 3.052r+O(1)$.

  We show that a stronger form of Ryser's conjecture that was used to
  prove the $r=3$ case fails for all $r>3$.  We also prove
  a fractional version of the following stronger form of Ryser's
  conjecture: in an $r$-partite hypergraph there exists a set $S$ of
  size at most $r-1$, contained either in one side of the hypergraph
  or in an edge, whose removal reduces the matching number by $1$.



\bigskip
Keywords: intersecting hypergraph; covering number; Ryser's conjecture;
fractional cover.
\end{abstract}  

\section{Introduction}

For a hypergraph $H$ we use $|H|$ to denote the number of edges
(also called lines) and $|V(H)|$ for the number of vertices. A
hypergraph is {\it $r$-uniform} if every edge has $r$ vertices on
it. We use $\pp{r}$ to denote any $r$-uniform projective plane. In
the standard terminology of projective planes, $\pp{r}$ has {\it
order} $r-1$.

A {\em $k$-cover} of a hypergraph is a set of $k$ vertices meeting
every edge of the hypergraph.  The {\it covering number} $\tau(H)$
of a hypergraph $H$ is the minimum $k$ for which there is a
$k$-cover of $H$. A {\em matching} is a set of disjoint edges, and
the {\it matching number} $\nu (H)$ of a hypergraph $H$ is the
maximum size of a matching consisting of edges of $H$.  A hypergraph
with $\nu(H)=1$ is said to be {\it intersecting}.  An intersecting
hypergraph is {\it linear} (also called {\it almost disjoint}) if
each pair of distinct edges meets in exactly one vertex.  In an
$r$-uniform hypergraph $\tau \le r \nu$, since a cover can be
obtained from the union of all edges in a matching that is maximal
with respect to containment.  This bound is sharp, as shown by
$\pp{r}$, or by the union of disjoint copies of $\pp{r}$.  Sharpness
is also attained by many other examples, such as the set of all
subsets of size $r$ in a ground set of size $kr-1$,
which has $\nu=k-1$ and $\tau=(k-1)r$.

A hypergraph is {\it $r$-partite} if its vertex set $V$ can be
partitioned into $r$ sets $V_1,\dots, V_r$, called the {\it sides} of
the hypergraph, so that every edge contains precisely one vertex from
each side.  In particular, $r$-partite hypergraphs are $r$-uniform.
Ryser~\cite{ryser} conjectured the following:

\begin{conjecture}
In an $r$-partite hypergraph, $\tau\le(r-1)\nu$.
\end{conjecture}

Very little is known about this conjecture.  In
\cite{ryser3} it was proved for $r=3$. For $r=4,5$ it was shown in
\cite{pennyscott} that there exists $\epsilon>0$ such that that $\tau
<(r-\epsilon)\nu$ in every $r$-partite hypergraph.

There is only one family of $r$-partite hypergraphs known to attain
Ryser's bound: subhypergraphs of truncated projective planes. Denoted
by $\tpp{r}$, the truncated projective plane of uniformity $r$ is
obtained from $\pp{r}$ by the removal of a single vertex $v$ and the
edges containing $v$. The sides of $\tpp{r}$ are the sets of vertices
other than $v$ on the edges of $\pp{r}$ containing $v$.  To achieve equality in
Ryser's conjecture it is enough to take only a small proportion of the
edges of $\tpp{r}$.  Kahn~\cite{kahn} proved:

\begin{theorem}\label{t:random_pg}
A random set of $22r\log r$ lines in an $r$-uniform projective plane
satisfies $\tau =r$ with probability tending to $1$ as $r\rightarrow\infty$.
\end{theorem}

This implies that:

\begin{theorem}\label{t:random_pp}
A random set of $22r\log r$ lines in $\tpp{r}$ satisfies $\tau\ge r-1$
with probability tending to $1$ as $r\rightarrow \infty$.
\end{theorem}

Solving an old problem of Erd\H{o}s and Lov\'asz, Kahn~\cite{kahn} proved
that there exist $r$-uniform intersecting hypergraphs with linearly many
edges, satisfying $\tau=r$. Mansour\etal\cite{yuster} conjectured that
something similar is true in the $r$-partite case. They defined
$f(r)$ to be the smallest integer $k$ for which there exists an
$r$-partite intersecting hypergraph $H$ with $k$ edges and $\tau(H)\ge
r-1$, and conjectured that $f(r)\le O(r)$.
If the hypergraphs constructed
in \cite{kahn} were part of $\pp{r}$, then this result would imply a
linear bound on $f(r)$ for infinitely many values, but unfortunately
this is not the case.
It is not even clear whether $f(r)$ exists for all $r$, since it
is conceivable that there is no hypergraph with the required
properties.  If $r-1$ is a prime power then $\tpp{r}$ is known to exist, providing proof
 that $f(r)$ is defined, but examples for other $r$ were
previously unknown. The first case for which the existence of $f(r)$ was previously unknown is $r=7$. In \sref{s:f(r)} we prove that $f(7)$
exists and that in fact $f(7)=17$. We also calculate $f(6)$ and
improve the general lower bound on $f(r)$.
We show that for $r\le5$ all hypergraphs attaining $f(r)$ are linear.
In contrast, there are non-linear hypergraphs that achieve $f(6)$ and $f(7)$.
We finish \sref{s:f(r)} by stating a number of open problems.

In \sref{s:frac} we consider various possible strengthenings of
Ryser's conjecture. In particular, a conjecture specifying the form of
the desired cover, which in the intersecting case is that the cover of
size $r-1$ can be assumed to be contained either in a side or in an
edge.  We show that a ``biased'' version of the conjecture, which is
true for $r=3$, is false for larger $r$. However, its fractional
formulation is true in a strong sense that provides also a fractional
version of the above conjecture on the form of the covers. We also
prove a fractional version of a strengthening of Ryser's conjecture
suggested by Lov\'asz.

\section{How many edges  are needed to achieve $\boldsymbol{\tau \ge r-1}$?}\label{s:f(r)}

In this section we study the function $f(r)$, defined in the
introduction.  In particular, we establish the values of $f(6)$ and
$f(7)$ and improve the lower bound on $f(r)$ proved in \cite{yuster}.
It is likely that Ryser's conjecture (if true) is sharp for all values
of $r$, but so far this has been shown only for $r$ for which $r-1$ is
a prime power.  The example below shows sharpness for the first open
case, $r=7$.

\begin{equation}\label{e:tight7}
\begin{array}{lllllllll}
1111111&& 1235354&& 2313664&& 4412343&& 6142564\\
2154322&& 1344433&& 3514555&& 4551234&& \\
3332221&& 1424266&& 3655163&& 5123253&& \\
4325512&& 2222135&& 4136465&& 5361365&&
\end{array}
\end{equation}
Here each sequence describes one edge, where the $i$-th symbol in the
sequence indicates
which vertex is taken in the $i$-th side $V_i$. The above example has
$17$ edges and $42$ vertices, $6$ on each side. We used a computer
to check that it has no $5$-cover, from which it follows that
$\tau=6=r-1$.

\begin{remark}
Independently, Abu-Khazneh and Pokrovskiy \cite{AP}   showed that
$f(7)$ exists. The bound they obtained was  $f(7)\le 22$.
\end{remark}

Our next aim is to study $f(r)$ for some small values of $r$. 
A common concept will be the idea of a {\it greedy cover}, which
is a cover obtained iteratively by including a vertex of maximum
degree in the hypergraph induced by the lines that have not yet
been covered. Note that in an intersecting hypergraph $H$ with more
than one line there is always a vertex of degree at least $2$.
Hence there is always a greedy cover of size at most 
$\lceil|H|/2\rceil$. If we have information about the degrees
of vertices in $H$ we can usually find a smaller greedy cover.
The next few lemmas will also recur in our calculations.

\begin{lemma}\label{l:stndrdcnt}
  Let $H$ be an intersecting $r$-partite hypergraph with
  covering number $\tau$. Suppose $H$ has maximum degree no more than
  $4$ and for $i=1,2,3,4$ let $x_i$ denote the number of vertices of
  degree $i$ in $H$. Then
\begin{align}
x_1+x_4&\ge \binom{|H|}{2}+3r\tau-2r|H|\label{e:x1x4},\\
x_3+3x_4&\ge \binom{|H|}{2}+r\tau-r|H|\label{e:x3x4}.
\end{align}
In each of \eref{e:x1x4} and \eref{e:x3x4} equality holds if and only if
$H$ is linear and has exactly $\tau$ vertices on each side.
\end{lemma}

\begin{proof}
Since each side is a cover there are at least $\tau$ vertices in every side.
Therefore
\begin{equation}\label{points}
 x_1+x_2+x_3+x_4\ge r\tau.
\end{equation}
Counting the pairs $(v,e)$ such that vertex $v$ lies on line
$e$ yields,
\begin{equation}\label{v,e}
x_1+2x_2+3x_3+4x_4=r|H|.
\end{equation}
Also, every two edges meet, which requires that
\begin{equation}\label{intersect}
 x_2+3x_3+6x_4\ge \binom{|H|}{2}.
\end{equation}
Now summing \eref{points} and \eref{intersect} and subtracting
\eref{v,e} we get \eref{e:x3x4}.
Similarly, three times \eref{points} plus \eref{intersect} minus twice
\eref{v,e} gives \eref{e:x1x4}. In both cases, equality requires equality in
\eref{points} and \eref{intersect}. The former means that each side has
exactly $\tau$ vertices and the latter means that $H$ is linear.
\end{proof}

\begin{lemma}\label{l:sidecover}
  Let $V_1$ be one side of an $r$-partite intersecting hypergraph
  $H$. Suppose that $V_1$ contains $y_1$ vertices of degree $1$ and
  $y_2$ vertices of degree at least $2$. Then
  $(y_1+1)/2+y_2\ge\tau(H)$.
\end{lemma}

\begin{proof}
  The lines through vertices of degree $1$ in $V_1$ can be greedily
  covered by $\lfloor(y_1+1)/2\rfloor$ vertices.  The remaining lines of
  $H$ can be covered by the vertices of degree at least $2$ in $V_1$.
\end{proof}

\begin{lemma}\label{l:1fact}
  Let $r$ be odd and suppose that $H$ is an intersecting $r$-partite
  hypergraph satisfying $|H|\le r$ and $\tau(H)\ge(r+1)/2$.  Then
\begin{itemize}
\item $|H|=r$
\item $\tau(H)=(r+1)/2$.
\item Each side of $H$ consists of one vertex of degree $1$ and
$(r-1)/2$ vertices of degree $2$.
\item Each line of $H$ contains one vertex of degree $1$ and
$r-1$ vertices of degree $2$.
\item $H$ is linear.
\end{itemize}
\end{lemma}

\begin{proof}
  If $|H|\le r-1$ or if $H$ has a vertex of degree greater than $2$
  then $H$ has a greedy cover using at most $(r-1)/2$ vertices. Hence
  $|H|=r$ and the maximum degree in $H$ is $2$. Even so, there is a
  greedy $(r+1)/2$-cover, so $\tau=(r+1)/2$.  Given that $r$ is odd,
  each side has at least one vertex of degree $1$, so $x_1\ge r$.  If
  any line contains two vertices of degree 1 then it cannot meet the
  other $r-1$ lines without breaching the maximum degree, hence
  $x_1=r$. The claims about degree sequences of sides and of lines
  follow. Also, counting intersections we have $r(r-1)/2$ pairs of
  lines and $r(r-1)/2$ degree 2 vertices, so $H$ is linear.
\end{proof}

Although we will not need it, it is possible to be even more precise
about the structure of $H$ in \lref{l:1fact}. From what we have shown
so far, it is clear that an extra line could be added through all of
the degree 1 vertices.  We would then have a 2-regular linear
intersecting $r$-partite hypergraph $H'$ with $r+1$ lines. Such a
hypergraph corresponds to a 1-factorisation of the complete graph
$K_{r+1}$. Each line in $H'$ represents a vertex of $K_{r+1}$ and each
side of $H'$ represents a $1$-factor, with each vertex of $H'$
specifying a different pair of vertices of $K_{r+1}$. Moreover, if we
take any $1$-factorisation of $K_{r+1}$, it will build an $H'$ as just
described, from which we can remove any one line to get a hypergraph
$H$ satisfying the conditions in \lref{l:1fact}.

In \cite{yuster} it was shown that
$f(r)\ge (3-\frac{1}{\sqrt{18}})r(1-o(1))\approx 2.764r(1-o(1))$.
We next improve this asymptotic lower bound.

\begin{theorem}\label{t:mindeg}
  Suppose $H$ is an intersecting $r$-partite hypergraph with covering
  number $\tau$ and maximum degree $\Delta$. Then
\[
\Delta\ge
\begin{cases}
2&\text{if $|H|\ge2$},\\
3&\text{if $\tau> (r+1)/2$},\\
4&\text{if $\tau> 2r/3+1$},\\
5&\text{if $\tau\ge (25r+23)/32$}.\\
\end{cases}
\]
\end{theorem}

\begin{proof}
Treating the right hand side of \eref{e:x3x4} as a quadratic in $|H|$ we see
that it is minimised when $|H|$ is $r$ or $r+1$. Hence
$x_3+3x_4\ge r(2\tau-r-1)/2$, so $\Delta\ge3$ whenever $\tau>(r+1)/2$.
(This bound is best possible, as demonstrated by
\lref{l:1fact}).

\def\wasepsilon{s}

Suppose that $|V(H)|=r\tau+\wasepsilon$ for some $\wasepsilon\ge0$.
Applying \lref{l:sidecover} to each side of $H$ we discover that
$r\tau\le (x_1+r)/2+(|V|-x_1)$ which means that
$x_1\le r+2(|V|-r\tau)=r+2\wasepsilon$.

Assume that $\Delta<4$.
Incorporating
the $\wasepsilon$ error term into the derivation of \eref{e:x1x4} we find that
\begin{align*}
x_1\ge \binom{|H|}{2}+3r\tau+3\wasepsilon-2r|H|
\ge r(3\tau-2r-2)+3\wasepsilon,
\end{align*}
by again minimising the quadratic in $|H|$.
Therefore
$r\ge x_1-2\wasepsilon\ge r(3\tau-2r-2)$, which implies that $\tau\le 2r/3+1$.

Next assume that $\Delta<5$. As in the previous case, we strengthen
\eref{e:x1x4} to give
\[
x_4\ge \binom{|H|}{2}+3r\tau+3\wasepsilon-2r|H|-x_1
\ge \binom{|H|}{2}+3r\tau-2r|H|-r.
\]
There is some side of $H$ with at least $\mu=\lceil x_4/r\rceil$ vertices of
degree $4$ on it. Using these vertices in a greedy cover we find that
\[
\tau\le \mu + \big\lceil(|H|-4\mu)/2\big\rceil
= \big\lceil |H|/2\big\rceil-\mu
\le \big\lceil |H|/2\big\rceil- \frac{1}{r}\binom{|H|}{2}-3\tau+2|H|+1.
\]
Maximising the quadratic in $|H|$ for each of the two possible parities
of $|H|$, we find that $\tau\le 25r/32+11/16+1/(32r)<(25r+23)/32$.
\end{proof}

Hence, using a greedy algorithm that chooses a vertex of highest degree at each
step, we can find an $(r-2)$-cover for any intersecting hypergraph with at most
$(2\times\frac12+3\times\frac16+4\times\frac{11}{96}+5\times\frac{7}{32})r+O(1)$
edges.

\begin{corollary}\label{cy:f(r)bnd}
$f(r)\ge 293r/96+O(1) \ge 3.052r+O(1)$ as $r\rightarrow\infty$.
\end{corollary}

Of course, \tref{t:mindeg} can also be used to find lower bounds on $f(r)$
for specific values of $r$. For example, $f(8)\ge1+2+2+2+3+3+5=18$,
$f(9)\ge1+2+2+2+2+3+3+5=20$ and $f(10)\ge1+2+2+2+2+3+3+4+5=24$.
In \cite{yuster} small values of $f(r)$ were studied, including a proof
that $12\le f(6)\le 15$.  We now determine the value of $f(6)$.

\begin{theorem}\label{f6is13}
$f(6)=13$
\end{theorem}
\begin{remark}
This was proved independently in \cite{AP}.
\end{remark}

\begin{proof}
The following is a $6$-partite intersecting hypergraph with
$30$ vertices, $13$ edges and $\tau=5$.
\begin{equation}\label{e:tight6}
\begin{array}{lllllllll}
111111&& 444114&& 125334 &&  241535 &&  545421\\
222211&& 553315&& 213444 &&  351224\\
333131&& 143252&& 255153 &&  514233
\end{array}
\end{equation}

Since $f(6)\ge12$, it suffices to now show
that any $6$-partite intersecting hypergraph $H$ with $12$ edges can
be covered with $4$ vertices. Assume to the contrary that $\tau(H)\ge5$.

First suppose that $H$ has a vertex $v$ of degree $5$ or more.  Let
$H'$ be the hypergraph made from $H$ by removing the lines through
$v$. By assumption $\tau(H')\ge4$, so \tref{t:mindeg} shows that
$H'$ has a vertex $u$ of degree at least $3$. Together $u$ and $v$
cover at least $8$ of the lines of $H$. The remaining lines can be
covered in pairs, using at most two further vertices, contradicting
$\tau(H)\ge5$.

{}From now on, let $x_i$ be the number of vertices in $H$ of degree
$i$. If there were two vertices of degree $4$ in the same side of
$H$ then we can find a $4$-cover, as before.  Therefore, $x_4\le 6$.
If there were four vertices of degree $3$ in the same side, then
they would form a cover, again a contradiction.  Also, if there is a
vertex of degree $4$ in a side, there can be at most one vertex of
degree $3$ in the same side.  Therefore, $x_3\le18-2x_4$. It follows
that $x_3+3x_4\le 24$. By \lref{l:stndrdcnt}, we find that
$x_3+3x_4=24$ and $H$ is linear. Moreover, $x_4=6$. Now, no two
vertices between them cover $8$ lines, since otherwise the remaining
lines could be covered greedily in a $4$-cover. It follows that each
pair of vertices of degree $4$ lie on a common line. There are
$\binom{6}{2}=15$ such pairs and only $12$ lines, so there are three
vertices of degree $4$ lying on a common line. These three vertices
will cover $10$ lines between them since $H$ is linear. The
remaining two lines can be covered by a single vertex, so we are
done.
\end{proof}

In \cite{yuster} it was shown that $f(7)\ge 14$. We next establish
the exact value of $f(7)$.

\begin{theorem}
$f(7)=17$.
\end{theorem}

\begin{proof}
  Suppose that $H$ is a $7$-partite intersecting hypergraph with
  $\tau(H)\ge 6$.  In \eref{e:tight7} we gave an example showing that
  $f(7)\le 17$, so it suffices to show that $|H|\ge17$. By
  \cite{yuster} we know that $|H|\ge 14$. Aiming for a contradiction,
  we assume that $|H|\le16$.

Let $v$ be a vertex of maximum degree in $H$. Removing the edges through
$v$ and all resulting isolated vertices we obtain a hypergraph
$H'$. Let $u$ be a vertex of maximum degree in $H'$.  From $H'$,
remove the edges through $u$ and all resulting isolated vertices to
obtain a hypergraph $H''$. By construction $\tau(H')\ge5$,
and $\tau(H'')\ge4$. Also $v$ has degree $|H|-|H'|$ in $H$ and
$u$ has degree $|H'|-|H''|\le|H|-|H'|$ in $H'$.
By \lref{l:1fact}, $|H''|\ge7$ so $u$ has degree
at most $4$ in $H'$.
By \tref{t:mindeg}, $u$ has degree at least $3$ in $H'$, so $|H'|\ge10$.

\case{$|H''|=7$ and $|H|\le15$}

The structure of $H''$ is prescribed by \lref{l:1fact}. In particular,
$H''$ has $4$ vertices per side. Let $D$ be the sum of
the degrees of the vertices in $V(H)\setminus(V(H'')\cup\{u,v\})$. By
\lref{l:sidecover}, each side that does not contain $u$ or $v$ contributes
at least $3$ to $D$. There are $6$ such sides if $u$ and $v$ are on the same
side. Otherwise $u$ and $v$ are on different sides and those
sides each contribute at least $1$ to $D$. Therefore
$D\ge\min\{2\times1+5\times3,6\times 3\}=17$. Since $H$ is
intersecting, each line from $H\setminus H''$ includes a cover of
$H''$ as well as either $u$ or $v$, and thus contributes at most $2$
to $D$. Hence $|H|-|H''|\ge9$, which is impossible.

\case{$|H''|\ge8$ and $|H|\le15$}

By \tref{t:mindeg} we may assume that $u$ has degree $3$ in $H'$, $v$
has degree $4$ in $H$ and $|H|=15$.  Applying \lref{l:stndrdcnt} we
find that $x_3+3x_4\ge42$ in $H$.  If there were two vertices of degree
$4$ on the same side we could use them as $v$ and $u$ and we would
end up in Case 1 above. So we may assume that
$x_4\le7$.  Thus either $x_3=21$ and $x_4=7$, or else
$x_3+x_4>28$. Either way, some side has $4$ vertices covering at least
$13$ lines or 5 vertices covering at least $15$ lines. Both options
lead to a $5$-cover.

\case{$|H|=16$ and $|H'|\le10$}

By the preliminary comments we know $|H'|=10$, $v$ has degree $6$
and $u$ has degree $3$.

Suppose $H'$ has $x_i$ vertices of degree $i$ for $1\le i\le3$.  By
\lref{l:stndrdcnt}, $x_1\ge10$ with equality only if every side has
precisely 5 vertices.  If we have equality then
there exists a side that contains at least two
vertices of degree 1, contradicting \lref{l:sidecover}. So
we may assume $x_1\ge 11$.
By the pigeonhole principle there is a line $e_1$ with at least two
vertices of degree $1$ on it.  Choose a line $e_2\ne e_1$, and suppose
it meets $e_1$ at a vertex $v_2$.  There are at least $10-3=7$ lines
that do not pass through $v_2$ and hence meet $e_2$ at some other
vertex. There are $6$ vertices in $e_2\setminus\{v_2\}$, so one of
them, say $v_3$, has degree $3$ and does not lie on $e_1$.  Now
consider the at least $5$ lines which do not pass through $v_2$ or
$v_3$. They have to meet $e_1$ in a vertex of degree greater than $1$
other than $v_2$. There are at most $4$ such vertices, so one of them,
say $v_4$, covers two of the lines as well as $e_1$. In other words,
$v_3$ and $v_4$ together cover $6$ lines, and the remaining lines can
be covered greedily with only two more vertices, a contradiction.

\case{$|H|=16$ and $|H'|\ge12$}

Consider any edge $e\in H$. The other 15 lines of $H$ must meet $e$,
so there is a vertex on $e$ of degree at least $4$.
Hence the degree of $v$, namely $|H|-|H'|$, must be $4$.
No side of $H$ can have a degree sequence containing $[4,4,4]$,
$[4,4,3,3]$ or $[4,3,3,3,3]$.
So $x_4\le 14$ and $x_3\le 35-2x_4$. Hence $x_3+3x_4\le 49$,
contradicting \lref{l:stndrdcnt}.

\case{$|H|=16$ and $|H'|=11$}

Here $v$ has degree $5$.
Each line through $v$ in $H$ includes a cover of $H'$ and hence
contains at most one vertex that is not in $V(H')\cup\{v\}$.  Hence
$|V(H)|\le|V(H')|+6$. By \eref{points} we have $|V(H)|\ge42$ so
$|V(H')|\ge36$.

Starting with the vertices of $H$ and the lines of $H'$, consider
adding the lines through $v$ one at a time in an arbitrary order. For
a given line through $v$, suppose that it includes $a_i$ vertices
(other than $v$ itself) that are of degree $i$ just before the line is
added. Given that $\tau(H')\ge5$ and that the lines through any two
vertices of degree $1$ can be covered by a single vertex, we see that
$a_1\le3$ and that if $a_0=1$ then $a_1\le1$. In other words,
$a_1+2a_0\le 3$. Similar reasoning can be applied to lines through $u$
that include $b_i$ vertices (other than $u$) of degree $i$ before they
are added to $H''$, showing that $b_1+2b_0\le5$.  These facts will be used
repeatedly in the subcases below.

\subcase{$|V(H')|=36$ and $H'$ has at least $8$ vertices of degree $1$}

It follows that $|V(H)|=42$, and this can only be achieved by each
line through $v$ having $a_0=1$ (and thus $a_1\le1$). It follows that $H$
has at least as many vertices of degree $1$ as $H'$. However, this
means that some side of $H$ contains two vertices
of degree $1$, contradicting \lref{l:sidecover}.

\subcase{$u$ has degree at most $3$}

Applying \lref{l:stndrdcnt} to $H'$, we see there is inequality in
\eref{points} and hence $x_3>13$. However, there cannot be 3 vertices
of degree 3 on one side (if there were, they could be used in a greedy
$4$-cover) so $x_3=14$, which in turn means that $|V(H')|=36$,
and $x_1=9$. Hence this case reduces to Case 5a.

\subcase{$u$ has degree $4$ and $|V(H')|=36$}


When $u$ has degree $4$, we have $|H''|=7$ so the
structure of $H''$ is prescribed by \lref{l:1fact}. In particular,
$|V(H'')|=28$ and $H''$ has one vertex of degree 1 on each side.

The only way that $H'$ can have 36 vertices is if 3 of the lines through
$u$ have $b_0=2$, $b_1\le1$ and the remaining line through $u$ has $b_0=1$,
$b_1\le3$. But this means that $H'$ has at least $7+7-6=8$ vertices of degree
$1$, so we are in Case 5a.

\subcase{$u$ has degree $4$ and $|V(H')|\ge37$}

Again $H''$ is prescribed by \lref{l:1fact}. However, this time each
line through $u$ has $b_0=2$, $b_1\le1$. Consequently $|V(H')|=37$ and
$H'$ has at least $7+8-4=11$ vertices of degree $1$.
Also at least $4$ of the lines through
$v$ have $a_0=1$, $a_1\le1$ and the remaining line through $v$ has
$a_1\le3$.
But this means that $H$ has at least $11+4-7=8$ vertices of degree $1$.
This is a contradiction unless $|V(H)|>42$, but that can only happen
if $|V(H)|=43$ and all lines through $v$ have $a_0=1$. In this case,
there are at least $11+5-5=11$ vertices of degree $1$, so at least
one side contradicts \lref{l:sidecover}.
\end{proof}

We say that a hypergraph achieves $f(r)$ if it is $r$-partite, has
$\tau\ge r-1$ and contains only $f(r)$ edges. It is notable that the
examples of hypergraphs achieving $f(r)$ that we gave in
\eref{e:tight7} and \eref{e:tight6} are not linear. We now contrast
this with the situation for smaller $r$.

\begin{theorem}\label{t:5lin}
For $r\le5$ the only hypergraphs achieving $f(r)$ are linear.
\end{theorem}

\begin{proof}
The statement is elementary to check for $r\le3$ so we will assume that
$r\in\{4,5\}$.

Suppose $H$ is a hypergraph achieving $f(r)$.
Suppose $H$ has $x_i$ vertices of degree $i$ for each $i$.
We assume that $H$ is not linear, so that we get inequality when we
apply \lref{l:stndrdcnt} to $H$.

\setcounter{case}{0}

\case{$r=4$}

In this case $|H|=f(4)=6$. There are no vertices of degree $4$ or
more, otherwise we would have a greedy $2$-cover.  By
\lref{l:stndrdcnt}, $x_3>3$. However, there cannot be two vertices of
degree $3$ on one side, since they would form a $2$-cover. So $x_3=4$
and each side has a vertex of degree $3$. Thus each side must have
degree sequence $[3,2,1]$ or $[3,1,1,1]$. If any side had the latter
option, there would be equality in \eref{intersect}. As we are
assuming $H$ is not linear, it follows that every side has degree
sequence $[3,2,1]$. Even so, there can only be a single pair of lines
that meets twice and all other pairs must meet once. Hence we can find
a vertex $v_3$ of degree $3$ such that the lines through $v_3$ are
disjoint apart from their intersection at $v_3$. Let $v_2$ and $v_1$
respectively be the vertices of degree $2$ and $1$ on the same side as $v_3$.
The union of the lines through $v_3$ contains every
vertex from $V(H)\setminus\{v_1,v_2\}$.

\begin{figure}[h]
\[
\includegraphics[scale=0.6]{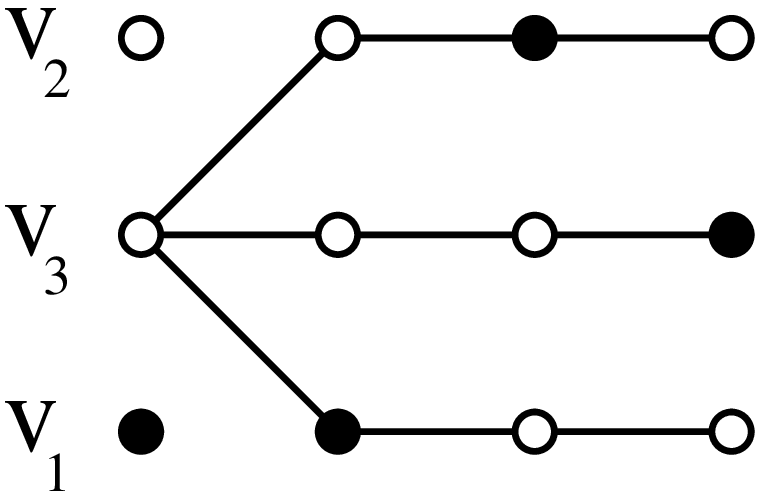}\qquad\qquad\qquad
\includegraphics[scale=0.6]{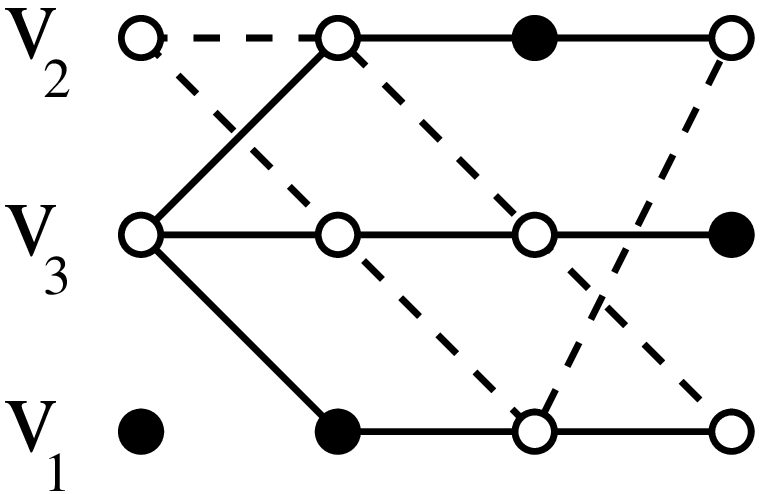}
\]
\caption{\label{f:vlines}}
\end{figure}

No line of $H$ contains two vertices of degree 1, since the total of
the degrees on a line must be at least $|H|+r-1=9$. Yet each side of
$H$ has a vertex of degree $1$, so there must be exactly one vertex of
degree $1$ on each line through $v_3$. Hence, up to isomorphism, the
lines through $v_3$ are as pictured on the left in \fref{f:vlines},
where vertices of degree 1 are shown as solid circles and vertices of
higher degree are hollow. By inspection, there is only one way to add
the lines through $v_2$, yielding the diagram on the right in
\fref{f:vlines}.  However, now the line through $v_1$ cannot meet all
the other lines, giving a contradiction.

\case{$r=5$}

In this case $|H|=f(5)=9$. There are no vertices of degree $5$ or more,
otherwise we would have a greedy $3$-cover.
By \lref{l:stndrdcnt}, $x_3+3x_4>11$. However, no side can have a degree
sequence containing $[4,3]$ or $[3,3,3]$, which means $x_4\ge2$.

Let $v$ be a vertex of degree $4$ in $H$. Removing the edges through
$v$ and all resulting isolated vertices we obtain a hypergraph $H'$
satisfying $\tau(H')\ge3$. The structure of $H'$ is dictated by
\lref{l:1fact}. Let $L$ denote the lines through $v$ in $H$. Since
$H$ is intersecting, each line of $L$ includes a cover of $H'$,
which means it has at most one vertex outside of $V(H')\cup\{v\}$.
At the same time each side of $H$ must have at least $4$ vertices,
since $\tau(H)\ge 4$. Therefore each line of $L$ contains a
different vertex outside of $V(H')\cup\{v\}$, which necessarily has
degree 1 in $H$. Now the only way to satisfy \lref{l:sidecover} is
if each line in $L$ contains a $3$-cover of $H'$ that includes a
vertex of degree $1$ in $H'$. Again, different lines in $L$ must
contain different such degree 1 vertices. Therefore, if two lines in
$L$ meet at a vertex other than $v$, that vertex has degree $2$ in
$H'$.

Let $u$ be a vertex of degree $4$ in $H$, other than $v$. By the
above, $u$ has degree $2$ in $H'$. Removing the edges through $u$ from
$H$ and all resulting isolated vertices, we obtain a hypergraph $H''$
which must also have the structure in \lref{l:1fact}.
Suppose the lines through $u$ in $H$ are
$\ell_1,\ell_2,\ell_3,\ell_4$, where $\ell_1,\ell_2\in H'$ and
$\ell_3,\ell_4\in L$.
Note that $\ell_3$ cannot meet $\ell_1$ or $\ell_2$ anywhere other
than at $u$, since $\ell_3$ only has 5 intersections with lines of
$H'$, counting multiplicities, and has to meet all 5 lines of
$H'$. A similar statement holds for $\ell_4$. Also $\ell_1$ and $\ell_2$
meet only at $u$ since $H'$ is linear. Finally, $\ell_3$ and $\ell_4$
do not meet at any vertex with degree $2$ in $H$. Putting these observations
together, we find that all vertices in
$V(H)\setminus\big(V(H'')\cup\{u\}\big)$ have degree $1$ in $H$. However,
$\ell_1,\ell_2,\ell_3,\ell_4$ between them contain at most 3 vertices of
degree $1$ in $H$. This gives the contradiction
$20=|V(H)|\le|V(H'')|+4=19$.
\end{proof}

Many questions remain open about $f(r)$. Mansour\etal\cite{yuster}
conjectured that it grows linearly. Since we have a linear lower bound
this is equivalent to:

\begin{conjecture}\label{cj:MSY}
$f(r)=O(r)$.
\end{conjecture}

However, that cannot be proved until a much more
fundamental question is answered.

\begin{problem}
For which $r$ is $f(r)$ defined?
\end{problem}

We have shown here that $f(7)$ is defined, but the issue
is unresolved for all $r>7$ for which $r-1$ is not a prime power.
One direction to approach \cjref{cj:MSY} is to try to
find infinitely many $r$ for which $f(r)$ is small. A natural
way to try to do this is to find small subsets of $\tpp{r}$
with $\tau=r-1$. It is fairly easy to see that approximately half of the
lines of a truncated projective plane can be deleted to get a sparser
hypergraph with the same $\tau$.  Recall that an {\it
  arc} of a projective plane is a set of vertices without three on a
line.  {\em Conics} show that there exist $(q+1)$-arcs in $PG(2,q)$,
called {\em ovals}.  Let $P_1,\dots,P_{r}$ be the vertices of the
oval, where $r=q+1$.  We delete $P_r$ and the lines through it, to
get an $r$-partite hypergraph, the truncated projective plane.  The
sides are identified with the deleted lines $P_rP_i$, where $1\le
i\le r-1$.  We delete the lines external to the oval.
That is, we keep the lines through $P_1,\dots,P_{r-1}$.  We kept
$r-2$ secants and 1 tangent through each of these $r-1$ vertices.
Therefore the number of remaining lines is $(q^2+q)/2=(r^2-r)/2$. We
claim that $\tau\ge r-1$.  Indeed, the degree of any of
$P_1,\dots,P_{r-1}$ is $r-1$, therefore the lines through $P_1$
cannot be blocked by $P_2,\dots,P_{r-1}$.  That is, either the $r-1$
lines through $P_1$ are blocked by different vertices or each of
$P_1,\dots,P_{r-1}$ is present in the cover.  In any case $\tau$ is
at least $r-1$.  This construction gives an easy way to show that
$f(4)\le6$ and $f(6)\le15$ as shown in \cite{yuster}. However,
\tref{t:random_pp} shows that it is far from optimal asymptotically.
So the challenge remains to find deterministic geometric
constructions that do much better, or indeed to show that the random
construction is essentially best possible.

\begin{problem}
How small can a subset of the lines of
$\tpp{r}$ be and still have $\tau=r-1$?
\end{problem}

Another variant is to insist that a hypergraph be linear, but not
necessarily a subset of a projective plane. It is not clear whether
being linear helps to achieve $f(r)$ or not.

\begin{problem}
Is $f(r)$ typically achieved by linear hypergraphs, non-linear hypergraphs
or both?
\end{problem}

In \tref{t:5lin} we saw that only linear hypergraphs achieve $f(r)$
for $r\le5$.  The examples that we gave in \eref{e:tight7} and
\eref{e:tight6} are not linear. However, the following is a
$6$-partite linear hypergraph with $\tau=5$:
\[
\begin{array}{lllllllll}
111111&& 212222&& 221333&& 322144&& 333213\\
413354&& 424412&& 432531&& 441245&& 514543\\
525251&& 543132&& 552315
\end{array}
\]
Clearly, $f(6)$ is achieved by both linear and non-linear hypergraphs.
We were not able to find a linear hypergraph achieving $f(7)$, and suspect
that no such hypergraph exists. Indeed, we were unable to answer the following
question:

\begin{problem}
Is there any linear intersecting $7$-partite hypergraph with $\tau=6$?
\end{problem}

Of course, the same question is interesting for other $r$ where $r-1$
is not a prime power. The analogous problem for $r$-uniform hypergraphs
is:

\begin{problem}
Is there any linear intersecting $7$-uniform hypergraph with $\tau=7$?
\end{problem}

Next we question the extent to which \tref{t:mindeg} generalises.

\begin{problem}
For each positive integer $d$ does there exist an $\epsilon>0$ such that,
for sufficiently large $r$, every $r$-partite intersecting hypergraph
with $\tau\ge(1-\epsilon)r$ has $\Delta\ge d$?
\end{problem}

Answering this might be one way to improve the lower bound on $f(r)$
given by \cref{cy:f(r)bnd}.

In studying $f(r)$ we have concentrated on the case $\nu=1$, but the
same questions can be asked for general $\nu$.  Let $f(r,k)$ be the
smallest number of edges in an $r$-partite hypergraph with $\nu=k$ and
$\tau\ge(r-1)k$. Note that $f(r,1)=f(r)$. Also, we can
remove edges from any hypergraph satisfying $\tau\ge(r-1)k$ to reach a
hypergraph with $\tau=(r-1)k$. It follows that any hypergraph achieving
$f(r,k)$ will have $\tau=(r-1)k$.  Taking $k$ disjoint copies of an
$r$-partite hypergraph with $\nu=1$ and $\tau=r-1$ shows that $f(r,k)
\le k f(r,1)$.  The results of \cite{HNS1,HNS2} imply that
$f(3,k) =k f(3,1)$, for all $k$.  Does equality hold more generally?

\begin{problem}
For which $r$ and $k$ is it true that $f(r,k) =k f(r,1)$?
\end{problem}

It would also be worth finding bounds or estimates for $f(r,k)$.

\section{Stronger versions and fractional covers}\label{s:frac}

In this section we consider various conjectures that would imply Ryser's
conjecture. We also consider versions involving the fractional covering
number $\tau^*$. In a fractional cover, each vertex is
assigned a non-negative real weight in such a way that the total
weight on each edge is at least $1$. The fractional covering number
$\tau^*$ is the least possible total of the vertex weights in a
fractional cover.

The first author has thought for some time that
the following stronger version of Ryser's
conjecture might be true for intersecting hypergraphs:

\begin{conjecture}\label{sideoredge}
  In an intersecting $r$-partite hypergraph $H$ there exists a side of
  size $r-1$ or less, or a cover of the form $e\setminus \{x\}$, for
  some $e \in H$ and $x \in e$.
\end{conjecture}

As we shall see shortly, a fractional version of \cjref{sideoredge} is
true.  A natural stronger version of \cjref{sideoredge} is that for
each side $V_i$ either $|V_i|<r$ or there exists an edge $e$ such that
$e \setminus V_i$ is a cover. However, this is false for $V_1$ in the
following example.  Let $H$ have a side $V_1$ of size $2^{r-2}$ and
sides $V_i=\{a_i,b_i\}$ for $i>1$.
The vertices $\{v_P\}$ of $V_1$ are indexed
by the subsets $P\subseteq\{2,\ldots,r\}$ that contain the element $2$.
For each such $P$ there are two edges, $\{v_P\} \cup \{a_i: i\in P\}
\cup \{b_i: i\notin P\}$ and $\{v_P\} \cup \{a_i: i\notin P\} \cup
\{b_i: i\in P\}$.

\cjref{sideoredge} for general $r$-partite hypergraphs is:

\begin{conjecture}\label{cj:}
  In an $r$-partite hypergraph $H$ with $\nu(H)=k$ there exist
  sets $S_1, \ldots, S_k$, each of size at most $r-1$ and contained in a side or in an
  edge, such that $\bigcup_{i \le k}S_i$ is a cover.
\end{conjecture}

Another direction of strengthening Ryser's conjecture is a ``biased'' version.
 For a set $S$ of vertices write $|S|_{bias}$ for
$|S \cap V_r| + |S\setminus V_r|/(r-1)$, where $V_r$ is the last side. In \cite{Zew12} the following was conjectured: in an  $r$-partite hypergraph $H$ with sides $V_1, \ldots, V_r$ there exists a
cover $C$ such that $|C|_{bias}  \le \nu(H)$. The motivation for this conjecture was that for $r=3$ this stronger version follows from the proof of the main result in \cite{ryser3}.
 A fractional version
was proved in two different ways in two theses of students
of the first author, \cite{orikfir} and \cite{nogazewi}.
Nevertheless, 
for  $r>3$ the conjecture is false.   The example showing it is a
well known one; the
family of cross-intersecting hypergraphs whose dual achieves the bound
in the biclique edge colouring conjecture of Gy\'arf\'as and Lehel
(see \cite{CFGLT}). For $i=1,\dots,r-2$ we take an edge $e_i$ that
uses the first vertex on side $r$ and the $i$-th vertex on side $j$
for $j=1,\dots,r-1$. Now for each permutation $\sigma$ of
$\{1,\dots,r-1\}$ add an edge that uses the second vertex on side $r$
and vertex $\sigma(j)$ on side $j$ for $j=1,\dots,r-1$.  Next, on each
of the first $r-1$ sides break vertex $r-1$ apart so that all lines
through it now go through a different vertex on that side.
Neither of the two vertices on the last side
are a cover on their own. Moreover, by \cite{CFGLT}, any cover that
avoids using a vertex from the last side has size at least
$2r-4>r-1$. Hence the ``biased'' conjecture fails for all $r>3$.

However, a fractional version is true, which yields also fractional versions of Conjecture \ref{cj:} and thus of Conjecture \ref{sideoredge}.
  An $r$-uniform hypergraph $H$ is said to be $(a,b)$-partitioned if
  $V(H)=V_1 \cup V_2$, where $V_1 \cap V_2 =\emptyset$, and
  $|e \cap V_1|=a$ and $|e \cap V_2|=b$ for every
  $e\in H$.

\begin{theorem}\label{t:fractionalstrong}
  Given a $(1,r-1)$-partitioned hypergraph with sides $V_1,\,V_2$
  there exist numbers $\beta_u \in \{0,1\}$ for each $u \in V_1$ and
  $\alpha_e \in \R^+$ for each $e \in H$, such that:
\begin{enumerate}
\item
$\sum_{u \in V_1} \beta_u + \sum_{e \in H} \alpha_e \le \nu(H)$, and:
\item
$\sum_{u \in V_1}\beta_u \chi_{\{u\}} + \sum_{e \in H} \alpha_e \chi_{e \setminus V_1}$ is a fractional cover for $H$.
\end{enumerate}
\end{theorem}

\begin{remark}
The theorem implies that in a $(1,r-1)$-partitioned hypergraph $\tau^* \le (r-1)\nu$. This was already known, since F\"uredi \cite{furedi}
proved this inequality for any $r$-uniform hypergraph not containing a copy of the $r$-uniform projective plane.
\end{remark}

For the proof we shall resort to topological
notions, in particular that of ``homological connectivity''.
A {\em simplicial complex} (or plainly a {\em complex})
is a closed downwards hypergraph, namely a collection of finite sets, called
``simplices'', containing with each simplex also all of its subsets.
The {\em homological connectivity} $\eta_H(X)$ of a complex $X$
is the minimal $k$ for
which all homology groups $H_i(X),~i \le k$, vanish, plus $2$ (the
addition of $2$ simplifies the formulation of several results). Intuitively,
$\eta_H(X)$ is the dimension of the smallest ``hole'' in $X$.  In
particular, $\eta_H \ge 1$ means ordinary connectivity of the
complex. For example, if $X$ is a $1$-dimensional complex (i.e., a
graph) that is a cycle, then there is a hole of dimension $2$, and no
hole of dimension $1$, and hence $\eta_H(X)=2$. We refer the reader to
\cite{ab, abm, meshulam} for some basic facts about connectivity.

Given a family $\ca=(A_1, \ldots ,A_m)$ of sets, a set formed by a partial choice
function from the $A_i$'s is said to be a {\em (partial) rainbow set}. A complete rainbow set is called a {\em transversal}.
Given a  complex $\cc$ on $\bigcup_{i \le m }A_i$  a rainbow set belonging to $\cc$ is called an {\em $(\ca,\cc)$-transversal}.
The maximal size of a $(\ca,\cc)$-transversal is denoted by $\nu(\ca,\cc)$.
We define the topological deficiency $\defic(\ca,\cc)$ as the maximum
of $|I|-\eta_H(\cc[\bigcup_{i \in I}A_i])$ over all $I \subseteq [m]$.
There is a topological deficiency version of Hall's theorem \cite{ah}:

\begin{theorem}\label{t:defic1}

$\nu(\ca,\cc) \ge m-\defic(\ca,\cc)$.
\end{theorem}

\begin{theorem}\label{t:defic2}
   If
  $\eta_H(\cc[\bigcup_{i \in I}A_i]) \ge |I|-d$
  for all $I \subseteq [m]$ then there exists a partial
  rainbow set belonging to $\cc$ of size $m-d$.
\end{theorem}

The topological Hall theorem is the case $d=0$.  It appears in
\cite{ah} in a homotopical version (and even this, only implicitly),
and explicitly in \cite{meshulam}. The case of general $d$ is obtained
by the familiar device of adding ``leeway''.

For any subset $S$ of $V$, we denote by $\chi_S$ the characteristic
function of $S$.  We shall need another definition, about a special
type of fractional covers. Let
$$\tau_s(H)=\min \Big\{\sum_{e \in H} \alpha_e: \sum \alpha_i\chi_e
\text{ is a cover for }H\Big\}.$$

A result connecting these concepts is:

\begin{theorem}\label{t:abm}\cite{abm}
  Let $H$ be a hypergraph and let $L(H)$ be its line graph. Then
  $\eta_H(\ci(L(H)) \ge \tau_s(H)$.
\end{theorem}

With the preliminaries at hand, we can now prove
\tref{t:fractionalstrong}.

\begin{proof}
Let $V_1=\{v_1,v_2,\ldots,v_m\}$.
For $i \le m$ let $A_i=\big\{f \subseteq V_2 \mid \{v_i\} \cup f \in H\big\}$,
and let $K= \bigcup_{i \le m}A_i$. Let $\cc=\ci(L(K))$. So,
the vertices of $\cc$ are the $(r-1)$-tuples belonging to $K$ and the
simplices are the matchings. Denote by $\Gamma$ the resulting ISR
system. Clearly, $\nu(\Gamma)=\nu(H)$.  Write $d$ for
$\defic(\Gamma)$. By \tref{t:defic1}, $\nu(\Gamma) \ge m-d.$ Let $J$
be a subset of $[m]$ such that $d=|J|-\eta_H(\cc[\bigcup_{i \in J}A_i])$.
By \tref{t:abm} we have $\tau_s(\bigcup_{i \in J}A_i) \le|J|-d$,
so there are numbers $\alpha_e,~~e \in K$ such that
$\sum_{e\in K}\alpha_e \le |J|-d$ and $\sum_{e \in K}\chi_e$ is a
fractional cover for $\bigcup_{i \in J}A_i$. Taking
$\beta_{u_i}=1$ for $i \notin J$ and  $\beta_{u_j}=0$ for $j \in  J$
completes the proof of the theorem.
\end{proof}

In particular, if $\nu(H)=1$ then \tref{t:fractionalstrong} says that
either $|V_1|=1$ or there exists a fractional cover of size at most
$r-1$, consisting of a linear combination with positive coefficients
of characteristic functions of sets of the form $e \setminus V_1$.

\bigskip

Another stronger version of Ryser's conjecture, conjectured
independently by Lov\'asz \cite{lovaszthesis} at around the same time
as Ryser made his conjecture, is:

\begin{conjecture}\label{lovasz}
  In an $r$-partite hypergraph $H$ there exists a set $S$ of vertices
  of size at most $r-1$, such that $\nu(H-S) \le \nu(H)-1$.
\end{conjecture}

 The strengthening of \cjref{lovasz} along the lines of \cjref{sideoredge} is:

\begin{conjecture}\label{sideoredgegeneral}
  In an $r$-partite hypergraph $H$ there exists a set $S$ of size
  $r-1$ or less, contained in an edge or in a side, whose removal
  reduces the matching number by at least $1$.
\end{conjecture}

In the fractional case it is enough to take sets of the second type, those contained in an edge. To show this, we first prove a Lemma that is stronger than we actually need, but could be of independent interest.

\begin{lemma}\label{l:smallh}
  In every $r$-partite hypergraph $H$ there exists an optimal
  fractional cover in which at most one side has positive weight
  on all of its vertices.
\end{lemma}

\begin{proof}
  Associating each dimension of $\R^n$ with a vertex of $H$, let $Q$
  be the polytope in $\R^n$ defined by $\vec{w}\chi_e \ge 1$ for all
  $e \in H$. Then $\tau^*(H)=\min\{\vec{w}\cdot\vec{1} \mid \vec{w}
  \in Q\}$ and the minimum is attained at a vertex $\vec{u}$ of $Q$.
  Suppose that
  there exist two distinct sides $V_i,\,V_j$ of $H$ such that
  $u(v)>0$ for every $v \in V_i \cup V_j$. We claim that
  $|V_i|=|V_j|$. To see this, assume to the contrary that (say)
  $|V_i|<|V_j|$. Now choose a positive
  $\epsilon\le\min\{u(v) \mid v \in V_j\}$,
  and define $u'(v)=u(v)-\epsilon$ for
  $v \in V_j$, $u'(v)=u(v)+\epsilon$ for $v \in V_i$, and
  $u'(v)=u(v)$ for $v \not \in V_i \cup V_j$. Then $\vec{u}'$ is a
  fractional cover of smaller size, contradicting the minimality
  property of $\vec{u}$.  Having shown that $|V_i|=|V_j|$, we now take
  a number $\epsilon >0$ smaller than $\min\{u(v) \mid v \in V_i \cup
  V_j\}$, and note that $\vec{u}=({\vec{u}'+\vec{u}''})/{2}$, where
  $\vec{u}':=\vec{u}+\epsilon \chi_{V_i}-\epsilon \chi_{V_j}$ and
  $\vec{u}'':=\vec{u}-\epsilon \chi_{V_i}+\epsilon \chi_{V_j}$ are both
  fractional covers. This contradicts the fact that $\vec{u}$ is a
  vertex of $Q$.

 We have shown that at least $r-1$ of the sides $V_i$ of $H$ contain a
 vertex $v$ of $H$ for which $u(v)=0$.
\end{proof}

\begin{theorem}\label{lovaszfrac}
  In an $r$-partite hypergraph $H$ there exists a set
  $S=e\setminus\{v\}$ for some $v \in e \in H$, such that
  $\nu^*(H-S) \le \nu^*(H)-1$.
\end{theorem}

\begin{proof}
By \lref{l:smallh} there exists an
optimal fractional cover $u$ of $H$ and a vertex $v$ such that
$u(v)=0$. Let $e$ be any edge of $H$ containing $v$, and let
$S=e\setminus \{v\}$.  Since $u$ is a fractional cover and $u(v)=0$,
the weight of $u$ on $S$ is at least $1$.
Clearly, $u$ restricted to $V\setminus S$ is a
fractional cover for $H-S$, proving the theorem.
\end{proof}

\subsection*{Acknowledgement}

We thank Abu-Khazneh and Pokrovskiy \cite{AP} for sharing their paper
and their ideas with us.

  \let\oldthebibliography=\thebibliography
  \let\endoldthebibliography=\endthebibliography
  \renewenvironment{thebibliography}[1]{%
    \begin{oldthebibliography}{#1}%
      \setlength{\parskip}{0.4ex plus 0.1ex minus 0.1ex}%
      \setlength{\itemsep}{0.4ex plus 0.1ex minus 0.1ex}%
  }%
  {%
    \end{oldthebibliography}%
  }

\end{document}